\documentclass[a4 paper, 10pt]{article}
\usepackage{amsmath}
\usepackage{latexsym}
\usepackage{amsfonts}
\usepackage{amssymb}
\usepackage{amsthm}
\usepackage{amsbsy}
\usepackage{graphicx}
\usepackage{bbm}
\usepackage{color}
\usepackage{epstopdf}

\newtheorem{theorem}{Theorem}[section]
\newtheorem{lemma}{Lemma}[section]

\newtheorem{definition}{Definition}[section]

\newtheorem{corollary}{Corollary}[section]
\newtheorem{remark}{Remark}[section]

\newtheorem{example}{Example}[section]

\begin{document}

\title{\Large \bf  A stochastic representation for the solution of approximated mean curvature flow
}

\author{
Raffaele Grande 
\thanks{Cardiff School of Mathematics, Cardiff University, Cardiff, UK, e-mail: GrandeR@cardiff.ac.uk}
}

\maketitle

\begin{abstract}
\noindent
The evolution by horizontal mean curvature flow (HMCF) is a partial differential equation in a sub-Riemannian setting with application in IT and neurogeometry (see \cite{cisar}). Unfortunately this equation is difficult to study, since the horizontal normal is not always well defined. To overcome this problem the Riemannian approximation was introduced. 
In this article we define a stochastic representation of the solution of the approximated Riemannian mean curvature using the Riemannian approximation and we will prove that it is a solution in the viscosity sense of the approximated mean curvature flow, generalizing the result of \cite{dirr}.
\end{abstract}

\section{Introduction}
The evolution by mean curvature flow  (MCF) has been studied extensively and it has many applications in image processing and neurogeometry (see e.g. \cite{cisar}).
We say that a hypersurface evolves by MCF if it contracts in the normal direction with normal velocity proportional to its mean curvature  see e.g. \cite{Ecker} for  further details.
It is well-known that this evolution may develop singularities in finite time in the Euclidean and Riemannian setting (as in the case of the dumbbell, see \cite{Ecker} for further details). To deal with such a singularities, many generalised approaches to study this evolution have been developed. In particular in 1991, Chen, Giga and Goto \cite{che} and, independently  Evans and Spruck  \cite{ev} introduced the so called level set approach, which consists in studying the evolving hypersurfaces as  level sets of (viscosity) solutions of suitable associated nonlinear PDEs.
In this paper we are interested in a degenerate version of such an evolution, namely evolution by horizontal mean curvature flow (HMCF) and its Riemannian approximation: we consider a hypersurface embedded in a sub-Riemannian structure (Carnot-type geometry), then the evolution contracts in the direction of the so called horizontal normal proportionally to its horizontal curvature (see Section 2 for details). We consider the level set approach which is now associated to a parabolic PDE far more degenerate than in the standard case.

\vspace{0.5 cm}

This approach was developed by Cardaliaguet, Quincampoix and Buckdahn in \cite{car} and contemporaneously but independently by Soner and Touzi \cite{tou2}  for the standard (Euclidean) case and generalised then by Dirr, Dragoni and von Renesse in \cite{dirr} for the  case of HMCF, that there is a connection between this equation and a suitable stochastic optimal control problem. In the Euclidean setting the dynamic can be expressed using the definition of the It\^{o} integral while in the sub-Riemannian case we have to use the definition of the Stratonovich integral. Roughly speaking, in the last case the dynamic is far more complex because we have a deterministic part (related to first order derivatives induced by the chosen geometry) and a stochastic one (related to some second order derivatives induced by the chosen geometry). However, as in the case of the Heisenberg group, sometimes it is possible to find some simplification of this dynamic as remarked in \cite{gra}, making the dynamic similar to the Euclidean one.
It is well known that it is possible to generalize this equation using a Riemannian approximation, as e.g. in \cite{cisar}.

\vspace{0.5 cm}
The aim of the chapter is to find a stochastic representation of the viscosity solution of approximated Riemannian mean curvature flow, generalizing the result obtained by \cite{dirr}. The chapter is organised as follows: in Section 2 we define some preliminary concepts, in Section 3 we introduce the horizontal mean curvature flow and in Section 4 we approximate it using a Riemannian approximation and in Section 5 we will find a stochastic representation of the solution of approximated mean curvature flow.
\section{Preliminaries} \label{pre}
We now briefly recall some basic geometrical definitions which will be key for defining the evolution by HMCF  (for further details see Chapter 1).
For more definitions and properties on sub-Riemannian geometries we refer to \cite{mont} and also \cite{BLU} for the particular case of Carnot groups.
\begin{definition}
Let $M$ be a $N$-dimensional smooth manifold, we can define for every point $p$ a subspace of $T_{x} M$ called $\mathcal{H}_{x}$. We define the \emph{distribution} as $\mathcal{H} = \{ (x,v) | \ p \in M, \  v \in \mathcal{H}_{x} \}$.
\end{definition}
\begin{definition}
Let $M$ be a manifold and $X$,$Y$ two vector fields defined on this manifold and $f:M \rightarrow \mathbb{R}$ a smooth function, then we define the \emph{(Lie) bracket} between $X$ and $Y$ as $[X,Y](f)= XY(f) - YX(f)$.\\
Let us consider $\mathcal{X} = \{ X_{1}, \dots , X_{m} \}$ spanning some distribution $\mathcal{H} \subset TM$, we define the k-bracket as $\mathcal{L}^{(k)} = \{ [X,Y] | X \in \mathcal{L}^{(k-1)} \ \ Y \in \mathcal{L}^{(1)} \}$ with $i_{j} \in \{1, \dots, m \}$ and $\mathcal{L}^{(1)} = \mathcal{X}$.
The associated Lie algebra is the set of all brackets between the vector fields of the family
\begin{equation*}
\mathcal{L}(\mathcal{X}) := \{ [X_{i} , X_{j}^{(k)} ] | X^{(k)}_{j} \ \mbox{k-length bracket of $X_{1}, \dots X_{m}$} \ k \in \mathbb{N} \}.
\end{equation*}
\end{definition}
The definition of H\"{o}rmander condition is crucial in order to work with PDEs in sub-Riemannian setting, because it allows us to recover the whole tangent space for every point.
\begin{definition} [H\"{o}rmander condition]
Let $M$ be a smooth manifold and $\mathcal{H}$ a distribution defined on $M$. We say that the distribution is \emph{bracket generating}   if and only if, at any point, the Lie algebra $\mathcal{L}(\mathcal{X})$ spans the whole tangent space.
We say that a sub-Riemannian geometry satisfies the \emph{H\"{o}rmander condition} if and only if the associated distribution is bracket generating.
\end{definition}
\begin{definition}
Let $M$ be a smooth manifold and $\mathcal{H}=span \{ X_{1}, \dots , X_{m} \} \subset TM$ a distribution and $g$ a Riemannian metric of $M$ defined on the subbundle $\mathcal{H}$. A \emph{sub-Riemannian geometry} is the triple $(M, \mathcal{H} , g)$.
\end{definition}
\begin{definition}
Let $(M, \mathcal{H}, g)$ be a sub-Riemannian geometry and $\gamma:[0,T] \rightarrow M$ an absolutely continuous curve, we say that $\gamma$ is an \emph{horizontal curve} if and only if
\begin{equation*}
\dot{\gamma}(t) \in \mathcal{H}_{\gamma(t)}, \ \ \mbox{for a.e.} \ t \in [0,T],
\end{equation*}
or, equivalently, if there exists a measurable function   $h : [0,T] \rightarrow \mathbb{R}^{N}$ such that
\begin{equation*}
\dot{\gamma}(t) = \sum_{i=1}^{m} h_{i}(t) X_{i}(\gamma(t)), \ \ \mbox{for a.e.} \ t \in [0,T],
\end{equation*}
where $h(t)=(h_{1}(t), \dots  , h_{m}(t))$ and $X_{1}, \dots X_{m}$ are some vector fields spanning the distribution $\mathcal{H}$.
\end{definition}
\begin{example}[The Heisenberg group] \label{a2n}
The most significant  sub-Riemannian geometry is the so called Heisenberg group.  For a formal definition of the Heisenberg group and the connection between its structure as non commutative Lie group and its manifold structure we refer to \cite{BLU}.
Here we simply introduce the 1-dimensional  Heisenberg group as the sub-Riemannian structure induced  on  $\mathbb{R}^{3}$ 
by the vector fields
\begin{equation*}
\label{HeisVectorFields}
X_1(x)= \begin{pmatrix}
1\\ 0 \\ - \frac{x_{2}}{2}
\end{pmatrix}
\quad  \textrm{and} \quad
X_2=
\begin{pmatrix}
0\\1 \\ \frac{x_{1}}{2}
\end{pmatrix},
\quad \forall \; x=(x_1,x_2,x_3)\in \mathbb{R}^3.
\end{equation*} 

We observe that the associated matrix is given by\begin{equation}\sigma(x) = \begin{bmatrix} 1 & 0 & - \frac{x_{2}}{2} \\ 0 & 1 & \frac{x_{1}}{2} \end{bmatrix}.\end{equation}
The introduced vector fields satisfy the H\"ormander condition with step 2: in fact $[X_1,X_2](x)=\begin{pmatrix}0\\0\\1\end{pmatrix}$
for any $x\in \mathbb{R}^3$.

\end{example}

The H\"{o}rmander condition is crucial to state the following theorem.

\begin{theorem}(\cite{mont})[Chow] Let $M$  be a smooth manifold and $\mathcal{H}$ a bracket generating distribution defined on $M$.  If $M$ is connected, then there exists a horizontal curve joining any two given points of $M$.\end{theorem}

\subsection{Carnot type geometries}
From now on we consider only the case where the starting topological manifold $M$ is the Euclidean $\mathbb{R}^N$. Moreover, in this paper we will concentrate on sub-Riemannian geometries with a particular structure: the so called Carnot-type geometries.
\begin{definition}
Let us consider $(M, \mathcal{H}, g)$ a sub-Riemannian geometry. We say that $X_{1}, \dots , X_{m}$, $m<N$, are Carnot-type vector fields if the coefficients of $X_{i}$ are 0 for $j \in \{ 1, \dots , m \} \setminus \{ i \}$, the $i$-component is equal to 1 and the other $m-N$ components are polynomial in $x$.
\end{definition}
The previous structure allows us to consider an easy and explicit Riemannian approximation. Nevertheless the approach apply also to the case where this additional structure is not fulfilled.
This structure applies to a large class of geometries. The Heisenberg group introduced in Example \ref{a2n} is obviously a Carnot-type geometry.
Carnot groups (see \cite{BLU} for definitions and properties) are a very important class of  sub-Riemannian geometries with in  addition a non commutative Lie group structure associated.


For later use we also introduce the matrix associated to the vector fields $X_{1} , \dots , X_{m}$, which is the
$N\times m$ matrix defined as
\begin{equation*}
\label{sigma} 
\sigma(x)=[X_{1}(x), \dots , X_{m}(x)]^{T}.
\end{equation*}
\begin{example}
In the case of the Heisenberg group  introduced in Example \ref{a2n}, the matrix $\sigma$ is given by
\begin{equation*}
\sigma(x)= \begin{bmatrix} 1 & 0 & - \frac{x_{2}}{2} \\ 0 & 1 & \frac{x_{1}}{2} \end{bmatrix},
\quad
\forall \, x=(x_{1},x_{2},x_{3}) \in \mathbb{R}^{3}.
\end{equation*}
\end{example}
In general, for Carnot-type geometries, the matrix $\sigma$ assumes the following structure:
\begin{equation}
\label{matrixCarnot-type1s}
\sigma(x) = \begin{bmatrix} I_{m \times m} &  A(x_{1}, \dots x_{m}) \end{bmatrix}
\end{equation}
where the matrix $A(x_{1}, \dots , x_{m})$ is a $(N-m) \times m$  matrix depending only on the first $m$ components of $x$.\\

We now want to introduce the Riemannian approximation, which will be crucial for our results.\\
Let  us consider a distribution $\mathcal{H}$ spanned by the  Carnot-type vector fields $\{ X_{1}, \dots , X_{m} \}$ defined on $\mathbb{R}^{N}$ with $m<N$ and satisfying the H\"ormander condition. 
It is possible to complete the distribution $\mathcal{H}$ 
by adding  $N-m$ vector fields $ X_{m+1} , \dots, X_{N}$ in order to construct an orthonormal basis for all $x \in \mathbb{R}^{N}$, i.e. 
$$\textrm{Span} \big( X_{1}(x) , \dots , X_{m}(x),X_{m+1}(x) , \dots, X_{N}(x)\big)=T_x\mathbb{R}^N\equiv\mathbb{R}^N,
\;\forall\, x\in \mathbb{R}^N.$$
The geometry induced,   for all  $\varepsilon > 0$,
by the distribution $$\mathcal{H}_\varepsilon (x)= span \{ X_{1}(x) , \dots , X_{m}(x),
 \varepsilon X_{m+1}(x) , \dots , \varepsilon X_{N}(x)  \}, \ \ \mbox{$ \forall x \in \mathbb{R}^{N}$}$$
is called \emph{Riemannian approximation} of our starting sub-Riemannian topology.
 The associated matrix is now
\begin{equation} \label{aaaft}
\sigma_{\varepsilon}(x) = [ X_{1}(x), \dots X_{m}(x) , \varepsilon X_{m+1}(x) \dots , \varepsilon X_{N}(x)]^{T}.
\end{equation}
Note that $\det( \sigma_{\varepsilon}(x)) \neq 0$.\\
  Note that, in the case of Carnot-type geometries, we can always choose 
 $$
 X_i(x)=e_i, \quad
 \forall i=m+1,\dots,N 
  \quad
 \forall x\in \mathbb{R}^N,
$$
 where by $e_i$ we indicate the standard Euclidean unit vector with 1 at the $i$-th component.

\begin{example}[Riemannian approximation of $\mathbb{H}^{1}$]
In the case of the Heisenberg group  introduced in Example \ref{a2n}, the matrix associated to the 
 Riemannian approximation  is for every point $x=(x_1,x_2,x_3)$ given by
\begin{equation*}
\sigma_{\varepsilon}(x)= \begin{bmatrix} 1 & 0 & - \frac{x_{2}}{2} \\ 0 & 1 & \frac{x_{1}}{2} \\ 0 & 0 & \varepsilon \end{bmatrix}.
\end{equation*}
\end{example}
This technique is called Riemannian approximation since, as  $\varepsilon \rightarrow 0^+$, then the geometry induced by Riemannian approximation converges, in sense of Gromov-Hausdorff (see \cite{gro} for further details), to the original sub-Riemannian geometry (as shown, as example, in \cite{cisar}).

\section{Horizontal mean curvature evolution}

Given a smooth hypersurface $\Gamma$, we indicate by $n_E(x)$ the standard (Euclidean) normal to the hypersurface $\Gamma$ at the point $x$. Since the family of vector fields $ \mathcal{X}_{\varepsilon}= \{X_{1} , \dots , X_{m}, \varepsilon X_{m+1} , \dots, \varepsilon X_{N}\}$ span the whole of $\mathbb{R}^N$ at any point of $\Gamma$, then $n_E(x)$ can be written w.r.t. such a basis, i.e.
$$
n_E(x)=\frac{\sum_{i=1}^N\alpha_i(x)X^{\varepsilon}_i(x)}
{\sqrt{\sum_{i=1}^N\alpha^2_i(x)}}.
$$
where $X^{\varepsilon}_i$ are the elements of $\mathcal{X}_{\varepsilon}$.
The following definitions will be key for this paper  (see Chapter 4 for further details).
\begin{definition} \label{c6d10000}
Given a smooth hypersurface $\Gamma$, the \emph{horizontal normal} is the renormalized projection of the Euclidean normal on the horizontal space $\mathcal{H}_x$, i.e.
\begin{equation*}
n_{0}(x):= \frac{ \alpha_1(x)X_1(x)+\dots \alpha_m(x)X_m(x)}{\sqrt{\alpha^2_1(x)+\dots +\alpha^2_m(x)}}\in \mathcal{H}_x \subset \mathbb{R}^N.
\end{equation*}
With an abuse of notation we will often indicate by $n_{0}(x)$ the associated $m$-valued vector 
\begin{equation}
\label{normal_coordinateVector}
n_{0}(x)= \frac{ (\alpha_1(x), \dots , \alpha_m(x))}{\sqrt{\alpha^2_1(x)+\dots +\alpha^2_m(x)}}\in \mathbb{R}^m.
\end{equation}
\end{definition}
The main difference between the horizontal normal and a standard normal is that the first may not exist even for smooth hypersurfaces. In fact at some points the horizontal normal is not defined while the Euclidean one exists. These points are called \emph{characteristic points}.
\begin{definition}\label{CharacteristicPoints}
 Given a smooth hypersurface $\Gamma$, characteristic points  occur whenever $n_E(x)$ is orthogonal to the horizontal plane $\mathcal{H}_x$, then its projection on such a subspace vanishes, i.e.
 $$
 \alpha^2_1(x)+\dots +\alpha^2_m(x)=0.
 $$
\end{definition}
Note that these points do not exist in the associated Riemannian approximation, in fact whenever $\Gamma$ is smooth the normal is defined at any point, which means
$$
\sum_{i=1}^N\alpha_i^2(x)\neq 0, \quad \forall x\in \Gamma.
$$

We recall that, for every smooth hypersurface, the \emph{mean curvature} at the point $x \in \Gamma$  is defined as the Euclidean divergence of the Euclidean normal at that point.
Similarly, for every smooth hypersurface, we can now introduce the  horizontal mean curvature.
\begin{definition}\label{c6glelw}
\label{horizontalMCF}
Given a smooth hypersurface $\Gamma$ and a non characteristic point $x\in \Gamma$, the 
 \emph{horizontal mean curvature} is defined as the horizontal divergence of the horizontal normal, i.e.
$
k_{0}(x) = div_{\mathcal{H}} n_{0}(x),
$
where $n_0(x) $ is the m-valued vector associated to the horizontal normal (see \eqref{normal_coordinateVector}) while $div_{\mathcal{H}} $ is the divergence w.r.t. the vector fields $X_1,\dots, X_m$, i.e.
\begin{equation*}
k_{0}(x) = X_1\left(\frac{\alpha_1(x)}{\sqrt{\sum_{i=1}^m\alpha_i^2(x)}}\right)+
\dots+
X_m\left(\frac{\alpha_m(x)}{\sqrt{\sum_{i=1}^m\alpha_i^2(x)}}\right).
\end{equation*}
\end{definition}
Obviously the horizontal mean curvature is never defined at characteristic points, since there the  horizontal normal does not exist.

\begin{definition} \label{c6glo99}
Let $\Gamma_t$ be a family of smooth hypersurfaces in $\mathbb{R}^N$. 
We say that $\Gamma_{t}$ is an \emph{evolution by horizontal mean curvature flow} 
of $\Gamma$ if and only if $\Gamma_0=\Gamma$ and for any smooth horizontal  curve $\gamma: [0,T] \rightarrow \mathbb{R}^{N}$ such that $\gamma(t) \in \Gamma_{t}$ for all $t \in [0,T]$, 
the horizontal normal velocity $v_{0}$ is equal to minus the horizontal mean curvature, i.e.
\begin{equation} \label{sla}
v_{0}(\gamma(t)): = - k_{0}(\gamma(t)) n_0(\gamma(t)),
\end{equation}
where $n_{0}(\gamma(t))$ and $k_{0}(\gamma(t))$ as respectively the horizontal normal and the horizontal mean curvature defined by Definitions \ref{c6d10000} and \ref{c6glelw} at the point $\gamma(t)$.
\end{definition}
Note that Definition \ref{c6glo99} is never defined at characteristic points.

In this subsection we consider a level set manifold $\Gamma$ which is smooth. We now compute the horizontal normal and the horizontal curvature for  smooth hypersurface expressed as zero level set, i.e.
$$
\Gamma= \big\{x \in \mathbb{R}^{N}| u(x)= 0 \big\},
$$ 
for some smooth function $u:\mathbb{R}^N\to \mathbb{R}$. Then 
 the Euclidean normal is simply
$n_E(x)= \frac{\nabla u(x)}{| \nabla u(x) |}$, which implies that the horizontal normal can be expressed as
\begin{equation} \label{c6wi99}
n_{0}(x)= \left( \frac{X_{1} u (x)}{\sqrt{\sum_{i=1}^{m} (X_{i} u(x))^2}} , \dots , \frac{X_{m} u (x)}{\sqrt{\sum_{i=1}^{m} (X_{i} u(x))^2}} \right).
\end{equation}
Note that $(X_1u,\dots , X_mu)\in \mathbb{R}^m$ is the so called horizontal gradient. \\
Similarly we can then write the horizontal mean curvature as 

\begin{equation} \label{c6cru99}
k_{0}(x)= \sum_{i=1}^{m} X_{i} \left( \frac{X_{i}u (x)}{\sqrt{\sum_{i=1}^{m} (X_{i} u(x))^2}} \right) .
\end{equation}

Let $\Gamma_{t} = \{ (x,t) | u(x,t) = 0 \}$ where $u$ is $C^{2}$. Applying \eqref{c6wi99} and \eqref{c6cru99} to the Definition \ref{c6glo99} we obtain that $u$ solves the following PDE, which is 
\begin{equation} \label{horMCF5}
u_{t} = Tr( (\mathcal{X}^{2}u)^{*}) - \bigg< (\mathcal{X}^{2}u)^{*} \frac{\mathcal{X} u}{|\mathcal{X} u|} , \frac{\mathcal{X} u}{|\mathcal{X} u|} \bigg>
\end{equation}
where $\mathcal{X}u$ is the so called horizontal gradient, that is 
\begin{equation*}
\mathcal{X}u:= (X_{1}u , \dots , X_{m}u)^{T}
\end{equation*}
and $(\mathcal{X}^{2}u)^{*}$  is the symmetric horizontal Hessian, that is
\begin{equation*}
((\mathcal{X}^{2} u)^{*})_{ij}:= \frac{X_{i}(X_{j}u) + X_{j}(X_{i}u)}{2}.
\end{equation*}

We consider the equation found by Dirr, Dragoni and Von Renesse in \cite{dirr}
\begin{align}\label{elso}
F(x,p,S) & =  - Tr( \sigma(x) S \sigma^{T} (x) + A(x,p)) \nonumber \\ &+ \left< \left( \sigma(x) S \sigma^{T}(x) + A(x,p) \right) \frac{\sigma(x) p}{| \sigma(x) p|} , \frac{\sigma (x) p}{|\sigma(x) p|} \right>
\end{align}
where
\begin{equation*}
A(x,p) = \frac{1}{2}  < \nabla_{X_{i}}X_{j}(x) + \nabla_{X_{j}}X_{i}(x), p>.
\end{equation*}
We observe that the equation $F(x,p,S)$ is well defined and continuous outside the characteristic points and we define $\mathcal{V}= \{(x,p) \in \Gamma \times T_{x}\Gamma | \  \sigma(x)p=0 \}$. In this way we observe that 
\begin{equation*}
F: (\mathbb{R}^{2N} \setminus \mathcal{V}) \times Sym(N) \rightarrow \mathbb{R}.
\end{equation*}
We remark that the function $F$ has some points in which is discontinuous, As consequence, in order to work with viscosity solutions, we have to compute the upper and lower envelops of this function (for further details about envelopes see Chapter 3).
\begin{definition} \label{c6rem}
Let us consider a locally bounded function $u: \mathbb{R} \times [0,T] \rightarrow \mathbb{R}$. 
\begin{itemize}
\item The \emph{upper semicontinuous envelope} is defined as
{\small \begin{align*}
u^{*}(t,x) \! := \! \inf \{ v(t, x)|\ v \ \mbox{cont. and \ } v \geq u \} \! = \! \limsup_{r \rightarrow 0^{+}} \{ u(s, y)|  |y - x| \leq  r, |t - s| \leq r \}.
\end{align*}}
\item The \emph{lower semicontinuous envelope} is defined as
{\small \begin{align*}
u_{*}(t,x) &:= \sup \{ u(t, x)|\ u \ \mbox{cont. and \ } v \leq u \} \!  = \! \liminf_{r \rightarrow 0^{+}} \{ u(s, y)| |y - x| \leq  r, |t - s| \leq r \}.
\end{align*}}
\end{itemize}
\end{definition}

\begin{remark}
If the function $u: \mathbb{R}^{N} \times [0,T] \rightarrow \mathbb{R}$ is continuous then it holds true
\begin{equation*}
u_{*}(t,x) = u(t,x) = u^{*}(t,x), \ \  \mbox{for all} \ \ \ (t,x) \in [0,T] \times \mathbb{R}^{N}.
\end{equation*}
\end{remark}
\begin{remark}
Applying the Definition \ref{c6rem} to the function $F$ as defined in \eqref{elso} we obtain
\begin{equation*}
F^{*}(x,p,S)= \begin{cases} -Tr(\overline{S}) + \left< \overline{S} \frac{\sigma(x) p}{|\sigma(x) p|} , \frac{\sigma(x) p}{| \sigma(x) p|} \right>, \ \ | \sigma(x) p| \neq 0, \\ -Tr(\overline{S}) + \lambda_{max}(\overline{S}), \ \ \ \ \ \ \ \ \ \ \ \ \  \  | \sigma(x) p| = 0 \end{cases}
\end{equation*}
and 
\begin{equation}
F_{*}(x,p,S)= \begin{cases} -Tr(\overline{S}) + \left< \overline{S} \frac{\sigma(x) p}{|\sigma(x) p|} , \frac{\sigma(x) p}{| \sigma(x) p|} \right>, \ \ |\sigma(x) p| \neq 0, \nonumber \\ -Tr(\overline{S}) + \lambda_{min}(\overline{S}), \ \ \ \ \ \ \ \ \ \ \ \ \ \  | \sigma (x) p| = 0 \end{cases}
\end{equation}
where $\overline{S} = \sigma(x) S \sigma^{T}(x) + A(x,p)$ with $\lambda_{max}$ and $\lambda_{min}$ the maximum and the minimum eigenvalues of the matrix $\overline{S}$.
\end{remark}

\section{Approximated Riemannian mean curvature flow}
The Equation \eqref{horMCF5} can be approximated to a Riemannian mean curvature flow using the Riemannian approximation (as seen in Chapter 1). This leads the following generalizations of the definitions of horizontal normal and horizontal divergence (see Chapter 4 for further details).
\begin{definition} \label{c6nap1}
Given a smooth hypersurface $\Gamma$, the \emph{approximated Riemannian normal} is the renormalized projection of the Euclidean normal on the horizontal space $\mathcal{H}^{\varepsilon}_x$, i.e.
\begin{equation*}
n_{\varepsilon}(x):= \frac{ \sum_{i=1}^{m}\alpha_i(x)X_i(x)+ \varepsilon \sum_{i=m+1}^{N} \alpha_{i}(x) X_{i}(x)}{\sqrt{\alpha^2_1(x)+\dots +\alpha^2_m(x) + \varepsilon^{2} \alpha^2_{m+1}(x) + \dots + \varepsilon^{2} \alpha^{2}_{N}(x)}}\in \mathcal{H}_x\subset \mathbb{R}^N.
\end{equation*}
With an abuse of notation, we will often indicate by $n_{\varepsilon}(x)$ the associated $N$-valued vector 
\begin{equation}
\label{normal_coordinateVector99}
n_{\varepsilon}(x)= \frac{( \alpha_{1}(x), \dots , \alpha_{m}(x), \varepsilon  \alpha_{m+1}(x), \dots, \varepsilon \alpha_{N}(x))^{T}}{\sqrt{\alpha^2_1(x)+\dots +\alpha^2_m(x) + \varepsilon^{2} \alpha^{2}_{m+1}(x) + \dots + \varepsilon^{2} \alpha^{2}_{N}(x)}}\in \mathbb{R}^N.
\end{equation}

\end{definition}

\begin{definition} \label{c6dap1}
Given a smooth hypersurface $\Gamma$ and a point $x\in \Gamma$, the \emph{approximated Riemannian mean curvature} is defined as the horizontal divergence of the approximated Riemannian normal, i.e.
$
k_{\varepsilon}(x) = div_{\mathcal{H}^{\varepsilon}} n_{\varepsilon}(x),
$
where $n_{\varepsilon}(x) $ is the $N$-valued vector associated to the horizontal normal (see \eqref{normal_coordinateVector99}) while $div_{\mathcal{H}} $ is the divergence w.r.t. the vector fields $X_1,\dots, X_m, \varepsilon X_{m+1}, \dots ,\varepsilon X_{N}$, i.e.
\begin{align}
k_{\varepsilon}(x) \! = \! \sum_{i=1}^{m} X_i\left(\frac{\alpha_i(x)}{\sqrt{\sum_{j=1}^m\alpha_j^2(x) + \varepsilon^{2} \sum_{k=m+1}^{N} \alpha_{k}^{2}(x)}}\right) \nonumber \\ +
\varepsilon \! \sum_{i=m+1}^{N} \! X_i\left(\frac{\varepsilon \alpha_i(x)}{\sqrt{\sum_{j=1}^m\alpha_j^2(x) + \varepsilon^{2} \sum_{k=m+1}^{N} \alpha_{k}^{2}(x)}}\right)\! .
\end{align}
\end{definition}
\begin{remark}
In this setting we do not have characteristic points on the hypersurface $\Gamma$.
\end{remark}
We define now the approximated Riemannian mean curvature flow.
\begin{definition} \label{c6gl}
Let $\Gamma_t$ be a family of smooth hypersurfaces in $\mathbb{R}^N$. We say that $\Gamma_{t}$ is an \emph{evolution by approximated Riemannian mean curvature flow} of $\Gamma$ if and only if $\Gamma_0=\Gamma$ and for any smooth horizontal curve $\gamma_{\varepsilon}: [0,T] \rightarrow \mathbb{R}^{N}$ such that $\gamma(t) \in \Gamma_{t}$ for all $t \in [0,T]$, the horizontal normal velocity $v_{\varepsilon}$ is equal to minus the horizontal mean curvature, i.e.
\begin{equation*} \label{sla}
v_{\varepsilon}(\gamma(t)): = - k_{\varepsilon}(\gamma_{\varepsilon}(t)) n_{\varepsilon}(\gamma_{\varepsilon}(t)),
\end{equation*}
where $n_{\varepsilon}(x(t))$ and $k_{\varepsilon}(x(t))$ as respectively the horizontal normal and the horizontal mean curvature defined by Definitions \ref{c6nap1} and \ref{c6dap1}.
\end{definition}
Developing all the computations following the example of \cite{dirr} we obtain the following partial differential equation
\begin{equation} \label{hver}
u_{t} = Tr(( \mathcal{X}_{\varepsilon}^{2} u)^{*}) - \left< (\mathcal{X}^{2}_{\varepsilon} u)^{*} \frac{\mathcal{X}_{\varepsilon} u}{|\mathcal{X}_{\varepsilon} u|} , \frac{\mathcal{X}_{\varepsilon} u}{|\mathcal{X}_{\varepsilon} u|} \right> = \Delta_{\varepsilon} u - \Delta_{0, \infty, \varepsilon} u,
\end{equation}
where
\begin{equation} \label{c6smas1a}
(\mathcal{X}^{2}_{\varepsilon} u )^{*}_{ij} = \frac{X_{i}^{\varepsilon}(X_{j}^{\varepsilon} u) + X_{j}^{\varepsilon}(X_{i}^{\varepsilon}u)}{2}.
\end{equation}
We observe now that we can write the Equation \eqref{hver} as
\begin{equation} \label{b}
u_{t} + F_{\varepsilon}(x,Du, D^2 u)=0,
\end{equation}
with
\begin{align} \label{pppp}
F_{\varepsilon}(x,p,S)& = - Tr( \sigma_{\varepsilon}(x) S \sigma^{T}_{\varepsilon} (x) + A_{\varepsilon}(x,p)) \nonumber  \\ &+ \left< \left( \sigma_{\varepsilon}(x) S \sigma_{\varepsilon}^{T}(x) + A_{\varepsilon}(x,p) \right) \frac{\sigma_{\varepsilon}(x) p}{| \sigma_{\varepsilon}(x) p|} , \frac{\sigma_{\varepsilon} (x) p}{|\sigma_{\varepsilon}(x) p|} \right>
\end{align}
with
\begin{equation*}
(A_{\varepsilon})_{ij}(x,p) = \frac{1}{2} \big< \nabla_{X_{i}^{\varepsilon}} X_{j}^{\varepsilon} + \nabla_{X_{j}^{\varepsilon}} X_{i}^{\varepsilon} , p \big>.
\end{equation*}
\begin{remark}
Applying the Definition \ref{c6rem} to the function $F_{\varepsilon}$ as defined in \eqref{elso} we obtain
\begin{equation*}
F^{*}(x,p,S)= \begin{cases} -Tr(\overline{S}_{\varepsilon}) + \left< \overline{S}_{\varepsilon} \frac{\sigma_{\varepsilon}(x) p}{|\sigma_{\varepsilon}(x) p|} , \frac{\sigma_{\varepsilon}(x) p}{| \sigma_{\varepsilon}(x) p|} \right>, \ \ |p| \neq 0, \\ -Tr(\overline{S}_{\varepsilon}) + \lambda_{max}(\overline{S}_{\varepsilon}), \ \ \ \ \ \ \ \ \ \ \ \ \  \  \ \   | p| = 0 \end{cases}
\end{equation*}
and 
\begin{equation}
F_{*}(x,p,S)= \begin{cases} -Tr(\overline{S}_{\varepsilon}) + \left< \overline{S}_{\varepsilon} \frac{\sigma_{\varepsilon}(x) p}{|\sigma_{\varepsilon}(x) p|} , \frac{\sigma_{\varepsilon}(x) p}{| \sigma_{\varepsilon}(x) p|} \right>, \ \ |p| \neq 0, \nonumber \\ -Tr(\overline{S}_{\varepsilon}) + \lambda_{min}(\overline{S}_{\varepsilon}), \ \ \ \ \ \ \ \ \ \ \ \ \ \ \ \  |p| = 0 \end{cases}
\end{equation}
where $\overline{S}_{\varepsilon} = \sigma_{\varepsilon}(x) S \sigma_{\varepsilon}^{T}(x) + A_{\varepsilon}(x,p)$ with $\lambda_{max}$ and $\lambda_{min}$ the maximum and the minimum eigenvalues of the matrix $\overline{S}_{\varepsilon}$.
\end{remark}

\subsection{The approximated Riemannian stochastic control problem} \label{o}

Let us consider a family of smooth vector fields $\mathcal{X}=\{ X_{1}, \dots X_{m} \}$ and its Riemannian approximation  $\mathcal{X}_{\varepsilon}=\{ X_{1}, \dots , X_{m} , \varepsilon X_{m+1}, \dots , \varepsilon X_{N} \}$.

\begin{definition}
We define the \emph{horizontal Brownian motion} the solution of the process
\begin{equation*}
d \xi = \sum_{i=1}^{m} X_{i}( \xi) \circ dB^{i}_{m},
\end{equation*}
where $B_{m}$ is a $m$-dimensional Brownian motion, $\circ$ the Stratonovich differential and $X_{i}$ the vector fields of $\mathcal{X}$ which span the distribution $\mathcal{H}$. We define the \emph{Riemannian approximated horizontal Brownian motion} as 
\begin{equation*}
d \xi_{\varepsilon} = \sum_{i=1}^{N} X_{i}^{\varepsilon}( \xi_{\varepsilon}) \circ dB^{i}_{N}
\end{equation*}
where $B_{N}$ is an $N$-dimensional Brownian motion and $X_{i}^{\varepsilon}$ the vector fields of $\mathcal{X}_{\varepsilon}$ which span the distribution $\mathcal{H}_{\varepsilon}$.
\end{definition}

Let $(\Omega, \mathcal{F}, \{ \mathcal{F}_{t} \}_{t \geq 0}, \mathbb{P})$ be a filtered probability space, $B_{j}$ is a $j$-dimensional Brownian motion adapted to the filtration $\{ \mathcal{F}_{t} \}_{t \geq 0}$ with $j=m,N$,
we recall that a \emph{predictable} process is a time-continuous  stochastic process $\{\xi(t)\}_{t \geq 0}$ defined on the filtered probability space $(\Omega, \mathcal{F}, \{\mathcal{F}_{t}\}_{t \geq 0}, \mathbb{P})$, measurable with respect to the $\sigma$-algebra generated by all left-continuous adapted process.
Given  a smooth function $g: \mathbb{R}^{N} \rightarrow \mathbb{R}$ (which parametrizes the starting hypersurface at time $t=0$) we introduce $\xi^{t,x, \nu}$ the solution of the stochastic dynamic
\begin{equation}\label{kkkkcap5}
\begin{cases} d \xi^{t,x, \nu}(s) = \sqrt{2} \sigma^{T} ( \xi^{t,x, \nu} (s)) \circ dB_{m}^{\nu}(s),  \  \  \ \ \ s \in (t,T], \\ dB_{m}^{\nu}(s)= \nu(s)dB_{m}(s), \\ \xi^{t,x, \nu}(t) = x, \end{cases}
\end{equation}
where the matrix $\sigma$ is defined in \eqref{matrixCarnot-type1s}, $\circ$ represents the differential in the sense of Stratonovich and
\begin{equation}\label{bladra}
\mathcal{A} = \big\{ \nu: [t,T] \rightarrow Sym(m) \ \mbox{predictable} \ | \nu \geq 0 , \ I_{m} - \nu^{2} \geq 0 , \ Tr(I_{m} - \nu^{2})=1\big\}
\end{equation}
and the function $V:[0,T] \times \mathbb{R}^{N} \rightarrow \mathbb{R}$ defined as
\begin{equation} \label{c6e200cc}
V(t,x):= \inf_{\nu \in \mathcal{A}}  ess \sup_{\omega \in \Omega} g(\xi^{t,x, \nu}(T)(\omega)).
\end{equation}
Similarly, for $\varepsilon>0$ fixed, we introduce $\xi^{t,x, \nu_{1}}_{\varepsilon}$ the solution of
\begin{equation}\label{kkk1f}
\begin{cases} d \xi^{t,x, \nu_{1}}_{\varepsilon}(s) = \sqrt{2} \sigma^{T}_{\varepsilon} ( \xi^{t,x, \nu_{1}}_{\varepsilon} (s)) \circ dB_{N}^{\nu_{1}}(s), \  \  \ \ \ s \in (t,T], \\ dB_{N}^{\nu_{1}}(s) = \nu_{1}(s) dB_{N}(s), \\    \xi^{t,x, \nu_{1}}_{ \varepsilon}(t) = x, \end{cases}
\end{equation}
where $\sigma_{\varepsilon}$ is the matrix defined in \eqref{aaaft} and
\begin{equation} \label{celfudra}
\mathcal{A}_{1} = \big\{ \nu_{1}: [t,T] \rightarrow Sym(N) \ \mbox{predictable} \ | \ \nu_{1} \geq 0 , \ I_{N} - \nu^{2}_{1} \geq 0 , \ Tr(I_{N} - \nu^{2}_{1})=1 \big\}
\end{equation}
and the function  $V^{\varepsilon}: [0,T] \times \mathbb{R}^{N} \rightarrow \mathbb{R}$ defined by
\begin{equation} \label{thr5}
V^{\varepsilon}(t,x) := \inf_{\nu \in \mathcal{A}_{1}}ess \sup_{\omega \in \Omega} g( \xi^{t,x, \nu_{1}}_{\varepsilon}(T)(\omega)).
\end{equation}
It is possible to show that the function $V$ solves in the viscosity sense respectively the level-set equation for the evolution by HMCF (see \cite{dirr}).

Note also that the sets of controls \eqref{bladra} and \eqref{celfudra} can be rewritten respectively   as
\begin{equation*} 
\mathcal{A} = \{ \nu^{2}| \  \nu \in \mathcal{A} \} = Co\{ I_{m} - a \otimes a | \  a \in \mathbb{R}^{m}, \ \ |a| =1 \},
\end{equation*}
and
\begin{equation*} 
\mathcal{A}_{1} = \{ \nu_{1}^{2}| \  \nu_{1} \in \mathcal{A}_{1} \} = Co\{ I_{N} - \overline{a} \otimes \overline{a} | \  \overline{a} \in \mathbb{R}^{N}, \ \ |\overline{a}| =1 \},
\end{equation*}
see \cite{car} for more details.\\
Next we introduce the $p$-regularising approximation of the functions $V$ and $V^{\varepsilon}$.
\begin{definition}\label{nor}
For $p>1$, the  $p$-value function  associated to the value function \eqref{c6e200cc} is defined as 
\begin{equation}\label{c6odi} 
V_{p}(t,x) := \inf_{\nu \in \mathcal{A}} \mathbb{E}[ |g(\xi^{t,x, \nu})(T)(\omega)|^{p} ]^{\frac{1}{p}},
\end{equation}
and, similarly, we can introduce the following $\varepsilon$-$p$-regularising function,
that is the  $p$-value function associated to the value function \eqref{thr5},
\begin{equation}\label{c6bar}
V^{\varepsilon}_{p}(t,x) := \inf_{\nu_{1} \in \mathcal{A}} \mathbb{E}[ |g(\xi^{t,x, \nu_{1}}_{\varepsilon})(T)(\omega)|^{p} ]^{\frac{1}{p}}.
\end{equation}
\end{definition}
\begin{definition}
We define the Hamiltonian associated to the horizontal stochastic optimal control problem \eqref{kkkkcap5} the function
\begin{equation*}
H(x,p,S) = \sup_{\nu \in \mathcal{A}} \bigg[ -Tr( \sigma(x) S \sigma^{T}(x) \nu^{2}(s)) + \sum_{i,j=1}^{m}( \nu^{2}(s))_{ij} \big< \nabla_{X_{i}} X_{j}(x) , p \big> \bigg].
\end{equation*}
where $\sigma$ is defined as in \eqref{matrixCarnot-type1s}, $p \in \mathbb{R}^{N}$ and $S \in Sym(N)$.
\end{definition}
\begin{definition}
We define the Hamiltonian associated to the approximated stochastic optimal control problem \eqref{kkk1f} the function
\begin{equation*}
H_{\varepsilon}(x,p,S) = \sup_{\nu_{1} \in \mathcal{A}_{1}} \bigg[ -Tr( \sigma_{\varepsilon}(x) S \sigma_{\varepsilon}^{T}(x) \nu_{1}^{2}(s)) + \sum_{i,j=1}^{N}( \nu^{2}_{1}(s))_{ij} \big< \nabla_{X^{\varepsilon}_{i}} X^{\varepsilon}_{j}(x) , p \big> \bigg].
\end{equation*}
where  $\sigma_{\varepsilon}$ is defined as in \eqref{a2n},  $p \in \mathbb{R}^{N}$ and $S \in Sym(N)$.
\end{definition}
\begin{remark}\label{emb}
The function $V_{p}$ solves in viscosity sense PDE:
\begin{equation}
\begin{cases} -(V_{p}) + H_{p}(x,DV_{p}, D^{2}V_{p}) =0, \ \ \ t \in [0,T) , \ x \in \mathbb{R}^{N},  \\ V_{p}(T,x)=g(x), \ \ \ \ \ \ \ \  \ \ \ \ \ \  \ \ \  \  \ \ \ \ \  x \in \mathbb{R}^{N} \end{cases}
\end{equation}
where
\begin{equation} \label{hami2}
H_{p}(x,q, M)  := \sup_{\nu \in \mathcal{A}} \bigg[-(p-1) r^{-1} Tr[\nu \nu^{T} q q^{T}] + Tr[\nu \nu^{T} M]\bigg],
\end{equation}
(see \cite{car} for further details). 
\end{remark}
\begin{remark}
Similarly to Remark \ref{emb}, for $\varepsilon>0$ and $p>1$ fixed, the function $V_{p}^{\varepsilon}$ solves in the viscosity sense the PDE
\begin{equation}
\begin{cases} -(V^{\varepsilon}_{p}) + H^{\varepsilon}_{p}(x,DV^{\varepsilon}_{p}, D^{2}V^{\varepsilon}_{p}) =0, \ \ t \in [0,T) , \ x \in \mathbb{R}^{N},  \\ V^{\varepsilon}_{p}(T,x)=g(x), \ \ \ \ \ \ \ \  \ \ \ \ \ \  \ \ \  \  \ \ \ \ \  x \in \mathbb{R}^{N} \end{cases}
\end{equation}
where
\begin{equation} \label{hami21}
H_{p}^{\varepsilon}(x,r,q, M) := H_{p}(x, r , q_{\varepsilon}, M_{\varepsilon}) = \sup_{\nu \in \mathcal{A}_{1}} \bigg[-(p-1) r^{-1} Tr[\nu \nu^{T} q_{\varepsilon} q^{T}_{\varepsilon}] + Tr[\nu \nu^{T} M_{\varepsilon}]\bigg],
\end{equation}
where $\mathcal{A}_{1}$ is given in \eqref{celfudra} and, for all $q \in \mathbb{R}^{N}$ and $M=(M_{ij})_{i,j=1}^{N} \in Sym(N)$,
\begin{equation*}
q_{\varepsilon}:= \begin{bmatrix} q_{1} \\ \dots \\ q_{m} \\ \varepsilon q_{m+1} \\ \dots \\ \varepsilon q_{N} \end{bmatrix} 
\end{equation*}
and
{ \small
\begin{equation*}
M_{\varepsilon}:= \begin{bmatrix} M_{11}& \dots &  M_{1m} & M_{1 (m+1)} &\dots& \varepsilon M_{1 N}\\ & & & \vdots & & &   \\ M_{m1} & \dots & M_{mm} & \varepsilon M_{(m+1) m}& \dots & \varepsilon M_{Nm}\\ \varepsilon M_{(m+1)1} & \dots & \varepsilon M_{(m+1)m} & \varepsilon^{2} M_{(m+1)(m+1)} & \dots &  \varepsilon^{2} M_{(m+1)N} \\ & & & \vdots & & &   \\ \varepsilon M_{1N} & \dots & \varepsilon M_{mN} & \varepsilon^{2} M_{(m+1)N} & \dots & \varepsilon^{2} M_{NN}     \end{bmatrix}.
\end{equation*}}
\end{remark}
\section{$V^{\varepsilon}$ as viscosity solution}
In this section we will prove the main result of this paper, but before doing it, we have to introduce some technical lemmas.
\begin{lemma} [Comparison Principle] \label{4a}
Let us consider $0< \varepsilon<1$ fixed. Let $g_{1}$, $g_{2}$ be continuous functions on $[0,T] \times \mathbb{R}^{N}$ with $g_{1}\leq g_{2}$ and $V_{i}^{\varepsilon}(t,x)$ for $i=1,2$ as defined in \eqref{thr5} with terminal costs $g_{i}$ then it holds true
\begin{equation*}
V_{1}^{\varepsilon}(t,x) \leq V^{\varepsilon}_{2}(t,x) \ \mbox{on $[0,T] \times \mathbb{R}^{N}$}.
\end{equation*}
\end{lemma}
\begin{proof}
It follows from the assumption $g_{1} \leq g_{2}$ and from the properties of infimum and essential supremum.
\end{proof}
\begin{lemma} \label{4b}
Let us consider $0< \varepsilon<1$ fixed. Let $g$ be a bounded and uniformly continuous function on $[0,T]\times \mathbb{R}^{N}$ and let $V^{\varepsilon}(t,x)$ be defined as in \eqref{thr5} with $g$ as terminal cost. Let us consider $\phi: \mathbb{R} \rightarrow \mathbb{R}$ continuous and strictly increasing. Then
\begin{equation*}
\phi(V^{\varepsilon}_{g}(t,x)) = V^{\varepsilon}_{\phi(g)}(t,x).
\end{equation*}
\end{lemma}
\begin{proof}
Since $\phi$ is an increasing and continuous function, we remark that $\phi(\inf A)= \inf \phi(A)$ where $A \subset \mathbb{R}$. Then, for every measurable function $f: \Omega \rightarrow \mathbb{R}$ it is easy to see that
\begin{equation*}
\phi(ess \sup f) = ess \sup(\phi(f))
\end{equation*}
and so we can conclude the proof.
\end{proof}
\begin{remark}
Lemmas \ref{4a} and \ref{4b} allow us to conclude that the set $\{V(t, x)\leq 0 \}$ depends only on the set $\{ g(x)\leq 0\}$ and not on the specific form of $g$. Furthermore we will show that $V^{\varepsilon}(t, x)$ solves (in the viscosity sense) the level set equation for the evolution by horizontal mean curvature flow for a fixed $0<\varepsilon<1$. 
\end{remark}
We state now the main theorem of the paper.
\begin{theorem} \label{c6wee}
Let us consider $0<\varepsilon<1$ fixed. Let $g:\mathbb{R}^{N} \rightarrow \mathbb{R}$ be  globally bounded and Lipschitz function, $T > 0$ and 
\begin{equation*}
\sigma_{\varepsilon}(x) = [X_{1}(x), ..., X_{m}(x), \varepsilon E_{m+1}(x) , \dots , \varepsilon E_{N}(x)]^{T}
\end{equation*}
a $N \times N$  matrix obtained from the Riemannian approximation of the $m \times N$ H\"{o}rmander  matrix $\sigma(x) = [X_{1}(x), ..., X_{m}(x)]^{T}$ with $m\leq N$ and smooth coefficients and $E_{i}=(0, \dots, 0,1,0 \dots, 0)^{T}$ where 1 is in the $i$-th component. Assuming that $\sigma_{\varepsilon}$ and $\nu_{\varepsilon}(x)=\sum_{i=1}^{N} \nabla_{ X^{\varepsilon}_{i}}X^{\varepsilon}_{j} (x)$ are Lipschitz (in order to have non-explosion for the solution of the SDE), then the value function $V^{\varepsilon} (t, x)$ defined by \eqref{thr5} is a bounded lower semicontinuous viscosity solution of the level set equation for the evolution by approximated Riemannian mean curvature flow, with terminal condition $V^{\varepsilon} (T, x) = g(x)$.
\end{theorem}
\begin{remark}
$V^{\varepsilon}(t,x)$ is a lower semicontinuous function.
\end{remark}
In order to prove the Theorem \ref{c6wee} we have to introduce the half-relaxed upper-limit.
\begin{definition} \label{c6rem1}
We define the \emph{relaxed half-relaxed upper-limit} of $V_{p}^{\varepsilon}(t,x)$
\begin{equation*}
V^{\sharp, \varepsilon}(t,x):=\limsup_{(s,y) \rightarrow (t,x) \ \ p \rightarrow \infty} V^{\varepsilon}_{p}(s,y).
\end{equation*}
\end{definition}
This lemma allows to use the definition of upper half-relaxed limit instead of the definition of upper envelope.
\begin{lemma} \label{c6say}
Let us consider $0<\varepsilon<1$ fixed. It holds true
\begin{equation*}
V^{\sharp , \varepsilon}(t,x) = V^{* , \varepsilon}(t,x) \ \ \  \mbox{for all} \ \ \ (t,x) \in [0,T] \times \mathbb{R}^{N}
\end{equation*}
where they are defined as in Definitions \ref{c6rem} and \ref{c6rem1}.
\end{lemma}
\begin{proof}
We observe that $V^{\sharp, \varepsilon} \geq V^{\varepsilon}$ and $V^{\sharp, \varepsilon}$ is upper semicontinuous function. Then, since $V^{*, \varepsilon}$ is the smallest upper envelope it holds $V^{\sharp, \varepsilon} \geq V^{*, \varepsilon}$. On the other hand, recalling that $V^{\varepsilon}_{p}(t,x) \leq V^{\varepsilon}(t,x)$ for any $t$,$x$, and $p>1$ and $\varepsilon>0$ fixed, then taking the $\limsup$ in $t$,$x$ and $p$ we obtain that $V^{\sharp , \varepsilon} \leq V^{*, \varepsilon}$ and as consequence the result follows.
\end{proof}
Another important observation is related to the $L^{p}$-norm related to $V^{\varepsilon}(t,x)$, i.e. $V^{\varepsilon}_{p}(t,x)$ as in Definition \ref{nor}.

We obtain the following result for $0<\varepsilon<1$ fixed.
\begin{lemma} \label{c6ca}
Let us consider $0<\varepsilon < 1$ fixed. Under the assumptions of Theorem \ref{c6wee}, we have 
\begin{equation*}
V^{\varepsilon}(t,x) = \lim_{p \rightarrow \infty} V_{p}^{\varepsilon}(t,x) \ \ \mbox{for all} \ \ (t,x) \in [0,T] \times \mathbb{R}^{N}
\end{equation*}
as pointwise convergence.
\end{lemma}
\begin{proof}
As the $L^{p}$ norm are bounded by essential supremum and increasing we obtain immediately for each fixed control and $\varepsilon>0$
\begin{equation*}
V^{\varepsilon}(t,x) \geq V^{\varepsilon}_{p}(t,x).
\end{equation*}
The other inequality will be proved as in \cite{dirr}. Let us consider $q\geq 1$, then by the property of the infimum we can find a control $\nu_{q}$ such that 
\begin{equation*}
\bigg( \mathbb{E}[ g^{p}( \xi^{t,x, \nu_{1, q}}_{ \varepsilon} (T))] \bigg)^{\frac{1}{q}} \leq V^{\varepsilon}_{q}(t,x) + \frac{1}{q}.
\end{equation*}
The controlled SDE \eqref{kkk1f} has a drift part which depends on the control only through $\nu_{1}^{2}$ (we recall by assumption that $\varepsilon>0$ is fixed) and our control set is convex in $\nu^{2}_{1}$. Proceeding as \cite{dirr}, we obtain that there exists a probability space $(\Omega , \mathcal{F} , \{ \mathcal{F}_{t} \}_{t \geq 0}, \mathbb{P} , B_{N}, \nu_{1} )$ such that for a subsequence $q_{k}$ the process $\xi^{t,x, \nu_{1, q_{k}}}_{ \varepsilon}$ converges weakly to $\xi^{t,x, \nu_{1}}$ and so for any fixed $\overline{q} \geq 1$
\begin{equation*}
\lim_{k \rightarrow \infty} \bigg( \mathbb{E}[ g^{q}(\xi^{t,x, \nu_{1, q_{k}}}_{ \varepsilon}(T))] \bigg)^{\frac{1}{q}} = \bigg( \mathbb{E} [g^{\overline{q}}(\xi^{t,x, \nu_{1}}_{\varepsilon}(T)) ]  \bigg)^{\frac{1}{\overline{q}}}.
\end{equation*}
Since the $L^{q}$ is non decreasing in $q$
\begin{equation*}
\bigg( \mathbb{E}[g^{q}(\xi^{t, x , \nu_{1}}_{\varepsilon}(T))] \bigg)^{\frac{1}{\overline{q}}} \leq \lim_{q \rightarrow \infty} V^{\varepsilon}_{q}(t,x).
\end{equation*}
Finally, using the convergence of $L^{q}$ norm to $L^{\infty}$ we obtain
\begin{equation*}
V^{\varepsilon}(t,x) \leq \lim_{q \rightarrow \infty} V_{q}^{\varepsilon}(t,x).
\end{equation*}
\end{proof}
In order to prove that $V^{\varepsilon}$ is a viscosity solution of approximated Riemannian mean curvature flow we have to recall a further lemma.
\begin{lemma}[\cite{car}] \label{c6noo}
Let $S \in Sym(N)$ such that the space of the eigenvectors associated to the maximum eigenvalue is of the dimension one. Then, $S \rightarrow \lambda_{max}(S)$ is $C^{1}$ in a neighbourhood of $S$. Moreover, $D \lambda_{max}(S)(H)= <Ha,a>$, for any $a \in \mathbb{R}^{m}$ eigenvector associated to $\lambda_{max}(S)$ and $|a|=1$.
\end{lemma}
The Theorem \ref{c6wee} is the consequence of the following theorem.
\begin{theorem}
Let us consider $0<\varepsilon<1$ fixed. Let $g: \mathbb{R}^{N} \rightarrow \mathbb{R}$ be a globally bounded and Lipschitz function, $T>0$ and $\sigma_{\varepsilon}(x)$ a Riemannian approximation of the $m \times N$-H\"{o}rmander  matrix $\sigma(x)$. Since the comparison principle holds (see \cite{az}), then the value function $V^{\varepsilon}(t,x)$ is the unique continuous viscosity solution of approximated Riemannian mean curvature flow, satisfying $V^{\varepsilon}(T,x)=g(x)$.
\end{theorem}

\begin{proof}
We divide this proof in two steps: we prove that $V^{\varepsilon}(t,x)$ is a viscosity supersolution and $V^{\varepsilon , \sharp}(t,x)$ is a viscosity subsolution.
\begin{itemize}
\item \emph{$V^{\varepsilon}$ is a viscosity supersolution}: Let us consider $\phi \in C^{1}([0,T]; C^{2}(\mathbb{R}^{N}))$ such that $V^{\varepsilon} - \phi$ has a local minimum at $(t,x)$. Two cases are possible:\\
if $\mathcal{X}_{\varepsilon} \phi(t,x)  \neq (0, \dots , 0)$ we have to verify that
\begin{equation*}
- \phi_{t}(t,x) - \Delta_{\varepsilon} \phi(t,x) + \Delta_{\varepsilon, \infty} \phi(t,x) \geq 0
\end{equation*}
where the equation is given as in \eqref{hver}. \\  If $\mathcal{X}_{\varepsilon} \phi(t,x)  = (0, \dots , 0)$ we have to verify that
\begin{equation*}
-\phi_{t}(t,x) - \Delta_{\varepsilon} \phi(t,x) + \lambda_{max}( (\mathcal{X}^{2}_{\varepsilon} \phi)^{*}  (t,x)) \geq 0
\end{equation*}
where $(\mathcal{X}^{2}_{\varepsilon} \phi)^{*}$ is defined as \eqref{c6smas1a}. \\
For any $p>1$ there exists a sequence $(t_{p}, x_{p})$ such that $V^{\varepsilon}_{p} - \phi$ has a local minimum at $(t_{p}, x_{p})$ and $(t_{p}, x_{p}) \rightarrow (t,x)$ a $p \rightarrow \infty$. In fact, we can always assume that $(t,x)$ is a strict minimum in some $B_{R}(t,x)$. Set $K= \overline{B_{\frac{R}{2}}(t,x)}$, the sequence of minimum points $(t_{p}, x_{p})$ converge to some $(\overline{t} , \overline{x}) \in K$. As $V^{\varepsilon}$ is the limit of $V^{\varepsilon}_{p}$ as $p \rightarrow \infty$ (see Lemma \ref{c6ca}) and lower semicontinuous, therefore by a standard argument yields that $(\overline{t}, \overline{x})$ is a minimum, hence it equals $(t,x)$. Then it holds true
\begin{equation*}
- \phi_{t}(t_{p}, x_{p}) + H_{\varepsilon}(x_{p}, (p-1) V_{p}^{-1} D \phi (D \phi)^{T} + D^{2} \phi)(t_{p}, x_{p}) \geq 0.
\end{equation*}
If $\sigma_{\varepsilon}(x) D \phi(t,x) \neq 0$, we can write the Hamiltonian in a more explicit way. Set
\begin{equation*} 
S_{1}= (p-1) V_{p}^{-1} (\mathcal{X}_{\varepsilon}\phi(t_{p}, x_{p}))( \mathcal{X}_{\varepsilon} \phi (t_{p}, x_{p}))^{T}
\end{equation*}
and
\begin{equation*} 
S_{2}=  (\mathcal{X}_{\varepsilon}^{2} \phi)^{*} (t_{p}, x_{p})
\end{equation*}
then
\begin{align}\label{5k}
H_{\varepsilon}(x_{p}, S_{1} ,S_{2}) & = - Tr(S_{1} + S_{2}) + \lambda_{max}(S_{1} + S_{2})  \nonumber   \\ & = -Tr(S_{1}) - Tr(S_{2}) + \lambda_{max}(S_{1} + S_{2})  \nonumber \\ & = -(p-1) (V_{p}^{\varepsilon})^{-1}(t_{p}, x_{p}) | \mathcal{X}_{\varepsilon} \phi (t_{p}, x_{p})|^{2}\nonumber  \\ &- \Delta_{\varepsilon} \phi(t_{p}, x_{p}) + \lambda_{max}(S_{1} + S_{2})
\end{align}
since the trace operator is linear and $Tr((\mathcal{X}_{\varepsilon} \phi(x_{p}))( \mathcal{X}_{\varepsilon} \phi (x_{p}))^{T} = | \mathcal{X}_{\varepsilon} \phi (x_{p})|^{2}$.
Now we use the Lemma \ref{c6noo} in order to expand the $\lambda_{max}$. We consider the matrix
\begin{equation*}
S= \frac{\mathcal{X}_{\varepsilon} \phi(t,x) (\mathcal{X}_{\varepsilon} \phi(t,x))^{T}}{V^{\varepsilon}(t,x)}
\end{equation*}
for which $\lambda_{max}(S)= \frac{|\mathcal{X}_{\varepsilon} \phi (t,x)|^{2}}{V^{\varepsilon}(t,x)}$ and where $a= \frac{\mathcal{X}_{\varepsilon}  \phi(t,x)}{|\mathcal{X}_{\varepsilon} \phi(t,x)|}$ since $\mathcal{X}_{\varepsilon} \phi(t,x) \neq 0$ (see \cite{car} for further remarks). Let us consider
\begin{equation*}
S_{p} = \frac{(\mathcal{X}_{\varepsilon} \phi (t_{p}, x_{p}))(\mathcal{X}_{\varepsilon} \phi (t_{p}, x_{p}))^{T}}{V^{\varepsilon}_{p}(t_{p},x_{p})},
\end{equation*}
it is immediate to observe that $S_{p}$ converges to $S$ as $p \rightarrow \infty$. By Taylor's formula we know that there exists a $\theta_{p} \in (0,1)$ such that
\begin{align*}
\lambda_{max} & \bigg( S_{p} + \frac{(\mathcal{X}_{\varepsilon}^{2} \phi)^{*} (t_{p} , x_{p})}{p-1} \bigg) = \lambda_{max}(S_{p}) \nonumber \\ & + \frac{1}{p-1} D \lambda_{max} \bigg( S_{p} + \frac{\theta_{p}}{p-1} (\mathcal{X}^{2}_{\varepsilon} \phi)^{*}(t_{p}, x_{p}) \bigg) (\mathcal{X}^{2}_{\varepsilon} \phi )^{*}(t_{p}, x_{p}).
\end{align*}
Using the fact that $\lambda_{max}$ is $C^{1}$ in a neighbourhood of $S$ and $S_{p} \rightarrow S$ to get 
\begin{align*}
\lambda_{max} &  \bigg( S_{p} + \frac{(\mathcal{X}_{\varepsilon}^{2}\phi)^{*} (t_{p}, x_{p})}{p-1} \bigg) = \lambda_{max}(S_{p}) \nonumber  \\ &+ \frac{1}{p-1}D \lambda_{max}(S)(\mathcal{X}^{2}_{\varepsilon} \phi )^{*} (t_{p}, x_{p}) + o\left( \frac{1}{p} \right)
\end{align*}
where $p o(1/p) \rightarrow 0$ when $p \rightarrow \infty$. As consequence we obtain
\begin{align*}
\lambda_{max}& \bigg( S_{p} + \frac{(\mathcal{X}^{2}_{\varepsilon} \phi )^{*} (t_{p}, x_{p})}{p-1} \bigg) \nonumber \\ & = \lambda_{max}(S_{p}) + \frac{< (\mathcal{X}^{2}_{\varepsilon} \phi )^{*} (t_{p},x_{p}) \mathcal{X}_{\varepsilon} \phi(t,x) , \mathcal{X}_{\varepsilon} \phi(t,x)>}{(p-1) | (\mathcal{X}_{\varepsilon} \phi)(t,x)|^{2}}
\end{align*}
then, expanding the $p$-Hamiltonian \eqref{5k} we obtain immediately the inequality,
If $\mathcal{X}_{\varepsilon} \phi(t,x)=0$ then we use the subadditivity of $S \rightarrow \lambda_{max}(S)$ and remark that, since $V_{p}^{\varepsilon}$ is supersolution
\begin{align*}
0  &\leq - \phi_{t} + H_{\varepsilon}(x_{p}, D\phi, (p-1) (V_{p}^{\varepsilon})^{-1} D \phi (D \phi)^{T} + D^{2} \phi) \nonumber \\ &\leq  - \phi_{t} - (p-1)(V_{p}^{\varepsilon})^{-1}| \mathcal{X}_{\varepsilon} \phi|^{2} - Tr( (\mathcal{X}^{2}_{\varepsilon} \phi )^{*})   \nonumber \\ &+ \lambda_{max}((p-1) (V^{\varepsilon}_{p})^{-1} \mathcal{X}_{\varepsilon} \phi ( \mathcal{X}_{\varepsilon} \phi)^{T} + (\mathcal{X}_{\varepsilon}^{2} \phi)^{*}) \nonumber \\ &\leq - \phi_{t} - (p-1)(V_{p}^{\varepsilon})^{-1}| \mathcal{X}_{\varepsilon} \phi|^{2} - Tr( (\mathcal{X}^{2}_{\varepsilon} \phi )^{*})   \nonumber \\ &+ (p-1) (V^{\varepsilon}_{p})^{-1} |\mathcal{X}_{\varepsilon} \phi|^{2} + \lambda_{max}(\mathcal{X}_{\varepsilon}^{2} \phi)^{*} \nonumber \\ &= - \phi_{t} - Tr((\mathcal{X}^{2}_{\varepsilon} \phi)^{*}) + \lambda_{max}(\mathcal{X}^{2}_{\varepsilon} \phi)^{*}.
\end{align*}
In the end, we can conclude now that $V^{\varepsilon}$ is a supersolution.
\item \emph{$V^{*, \varepsilon}$ is the subsolution}: As consequence of Lemma \ref{c6say} we can write $V^{*, \varepsilon} = V^{\sharp , \varepsilon}$. Let $\phi \in C^{1}([0,T]; C^{2}(\mathbb{R}^{N}))$ such that $V^{\sharp, \varepsilon} - \phi$ has a strict maximum at $(t_{0}, x_{0})$. Let us consider a sequence of maximum points of $V_{p}^{\varepsilon} - \phi$, we can find a subsequence converging to $(t,x)$. Since $V_{p}^{\varepsilon}$ is the solution of
\begin{equation} \label{c6mis}
\begin{cases} -(V_{p})_{t} + H_{\varepsilon}(x , DV_{p}^{\varepsilon}, (p-1) (V_{p}^{\varepsilon})^{-1} DV_{p}^{\varepsilon} (DV_{p}^{\varepsilon})^{T} + D^{2}V_{p}^{\varepsilon}) = 0 \\  \ \ \ \ \ \ \ \ \ \ \ \ \ \ \ \ \ \ \ \ \ \ \ \ \ \ \ \ \ \ \ \   \ \ \ \ \ \ \ \ \ \ \ \  \ \  \ \ \ \ \ \ \  \ \ \ \ \ x \in \mathbb{R}^{N}, \ t \in [0,T), \\ \ \ \ \ \  \ \  \ \ \ \ \ \  \ \ \ \ \ \ \ \ \ \ \ \ \ \ \ \ \ \ \ \ \ \ \ \ \ \ \ \ \ \   V^{\varepsilon}_{p}(T,x) = g(x), \ \ \ \ \  x \in \mathbb{R}^{N} \end{cases}
\end{equation}
then we have that
\begin{equation}\label{5p}
0 \leq - \phi_{t} + H_{\varepsilon}(x, (p-1) (V_{p}^{\varepsilon})^{-1} D \phi (D \phi)^{T} + D^{2} \phi )
\end{equation}
at the point $(t_{p}, x_{p})$. We define for any $z>0$, $x, d \in \mathbb{R}^{N}$and any $N \times N$ symmetric matrix $S$
\begin{align*}
H^{\varepsilon}_{p}(x,z,d,S) = - \frac{(p-1)}{z} | \sigma_{\varepsilon}(x) d |^{2} - Tr( \sigma^{T}_{\varepsilon}(x) S \sigma_{\varepsilon}(x) + A_{\varepsilon}(x,d)) \nonumber \\ \lambda_{max} \bigg( \frac{(p-1)}{z} (\sigma_{\varepsilon}(x) d) (\sigma_{\varepsilon}(x) d)^{T} + \sigma_{\varepsilon}^{T}(x)S \sigma_{\varepsilon}(x) + A_{\varepsilon}(x,d) \bigg)
\end{align*}
and
\begin{align*}
(H_{\varepsilon})^{*}(x,d,S) = \begin{cases} - Tr( \sigma_{\varepsilon}^{T}(x)S \sigma_{\varepsilon}(x) + A_{\varepsilon}(x,d)) \\  \ \ \  \ \ \ \ \  + \bigg< (\sigma^{T}_{\varepsilon}(x) S \sigma_{\varepsilon}(x) + A_{\varepsilon}(x,d)) \frac{\sigma_{\varepsilon}(x) d}{|\sigma_{\varepsilon}(x) d|}, \frac{\sigma_{\varepsilon}(x) d}{|\sigma_{\varepsilon}(x) d|} \bigg>, \\ \ \ \ \ \ \ \ \ \ \ \ \ \ \ \ \ \ \ \ \ \ \ \ \ \ \  \ \ \ \ \ \ \ \ \ \ \ \ \ \ \ \ \ \ \  \ \ \ \ \ \ \ \ \ \ \ \ \ \ \ \ \ |d| \neq 0, \\ - Tr( \sigma_{\varepsilon}^{T}(x)S \sigma_{\varepsilon}(x) + A_{\varepsilon}(x,d)) \\  \ \ \  \ \ \ \ \  + \lambda_{max} (\sigma^{T}_{\varepsilon}(x) S \sigma_{\varepsilon}(x) + A_{\varepsilon}(x,d)), \ \ \ \ \ \ \ \ \ \  |d|=0 \end{cases}
\end{align*}
and, as stated in \cite{dirr}, we can observe that
\begin{equation*}
H^{\varepsilon}_{p}(x,z,d,S) \geq (H^{\varepsilon})^{*}(x,d,S).
\end{equation*}
We remark that for $|d|=0$ is immediate, for $|d| \neq 0$ we observe that
\begin{align*}
\lambda_{max} \bigg( \frac{(p-1)}{z} (\sigma_{\varepsilon}(x) d) (\sigma_{\varepsilon}(x) d)^{T} + \sigma_{\varepsilon}^{T}(x)S \sigma_{\varepsilon}(x) + A_{\varepsilon}(x,p) \bigg) \nonumber \\  \geq  \frac{(p-1)}{z} |\sigma_{\varepsilon}(x) d|^{2} + \lambda_{max} ( \sigma_{\varepsilon}^{T}(x)S \sigma_{\varepsilon}(x) + A_{\varepsilon}(x,p) )
\end{align*}
and, called $\overline{S}_{\varepsilon}= \sigma_{\varepsilon}^{T}(x)S \sigma_{\varepsilon}(x) + A_{\varepsilon}(x,p)$ 
\begin{equation*}
\lambda_{max}(\overline{S}_{\varepsilon}) = \max_{|a|=1} <\overline{S}_{\varepsilon} a , a>
\end{equation*}
we obtain immediately the inequality. Let us consider $\varepsilon >0$, set $z= \phi^{-1}(t_{p}, x_{p})>0$, $d=D \phi (t_{p} , x_{p})$, $S= D^{2} \phi(t_{p}, x_{p})$, then taking the limsup of \eqref{5p} we obtain for $p \rightarrow \infty$ and recalling that, by definition, $(H_{\varepsilon})^{*} \geq (H_{\varepsilon})_{*}$ we obtain
\begin{equation*}
0 \geq \phi_{t} + (H_{\varepsilon})_{*}(x, D\phi ,D^{2}\phi)
\end{equation*}
at $(t,x)$. The result follows immediately.
\end{itemize}
\end{proof}
Now, in order to prove the main theorem of this section, we need a further lemma.
\begin{lemma} \label{5l}
Let us consider $0<\varepsilon <1$ fixed. For any $x \in \mathbb{R}^{N}$, $V^{\varepsilon, \sharp}(T,x) \leq g(x)$.
\end{lemma}
\begin{proof}
By contradiction, we assume that it is not true and that there exists a point $x_{0}$ such that $V^{\varepsilon , \sharp}(T,x) \geq g(x_{0}) + \delta$, for $\delta >0$ sufficiently small. We use as test function
\begin{equation*}
\phi(t,x) = \alpha(T-t) + \beta |x - x_{0}|^{2} + g(x_{0}) + \frac{\delta}{2}
\end{equation*}
with $\alpha > -C \beta$, with $C$ a constant depending just on the data of the problem and the point $x_{0}$ and $\beta>1$ sufficiently large.
We remark that
\begin{equation*}
\phi_{t}(t,x) = \alpha , \ \ \ \  D\phi(t,x)= 2 \beta (x-x_{0}) , \ \ \ \ D^{2} \phi (t,x) = 2 \beta Id.
\end{equation*}
We can find a sequence $(t_{k} , x_{k}) \rightarrow (T, x_{0})$ and $p_{k} \rightarrow \infty$ as $k \rightarrow \infty$ such that $V^{\varepsilon}_{p_{k}} - \phi$ has a positive local maximum at some point $(s_{k} , y_{k})$, for any $k>1$. To obtain the contradiction we use the fact that $V^{\varepsilon}_{p_{k}}$ is solution of the Equation (\ref{c6mis}) in order to obtain $\alpha + C \beta \leq 0$. We observe that the functions $V^{\varepsilon}_{p}$ are bounded uniformly in $p$ and $\varepsilon$ is fixed so, by the growth of $|x-x_{0}|$, the maximum points are such that $y_{k} \in \overline{B_{R}(x_{0})}=: K$ with $R$ independent of $k$. In the point $(s_{k}, y_{k})$ it holds true
\begin{align*}
0 & \geq \alpha - H_{\varepsilon}(y_{k}, (p-1) \phi^{-1} D \phi (D \phi)^{T} + D^{2} \phi) \nonumber \\ & \geq \alpha - 2 \beta Tr (\sigma_{\varepsilon}(y_{k}) \sigma^{T}_{\varepsilon}(y_{k}) + A_{\varepsilon}(y_{k}, y_{k} - x_{0})) \nonumber \\ &+ 2 \beta \lambda_{min} (\sigma_{\varepsilon}(y_{k}) \sigma^{T}_{\varepsilon}(y_{k}) + A_{\varepsilon}(y_{k}, y_{k} - x_{0})).
\end{align*} 
Then recalling that there is a compact set $K$ such that  $y_{k} \in K$ for all $k$, by continuity, we get $0 \geq \alpha + C \beta$, with
\begin{align*}
C&= - \max_{x \in K} Tr(\sigma_{\varepsilon}(x) \sigma_{\varepsilon}^{T}(x)) - \max_{x \in K} A_{\varepsilon}(x, x-x_{0}) \nonumber \\ &+ \min_{k \in K}\lambda_{min}(\sigma_{\varepsilon}(x) \sigma_{\varepsilon}^{T}(x)) + \min_{x \in K} \lambda_{min}(A_{\varepsilon}(x, x-x_{0}))
\end{align*}
with such estimate we obtain the contradiction, i.e. the thesis.
\end{proof}
\begin{corollary}
Let us consider $0<\varepsilon<1$ fixed. Let $g: \mathbb{R}^{N} \rightarrow \mathbb{R}$ be bounded and H\"{o}lder continuous, $T>0$ and $\sigma_{\varepsilon}(x)$ a $N \times N$-H\"{o}rmander matrix like in Theorem \ref{c6wee}. Since the comparison principle holds (see \cite{az}), then the value function $V^{\varepsilon}(t,x)$ is the unique continuous viscosity solution of the level set equation \eqref{hver}, satisfying $V^{\varepsilon}(T,x)=g(x)$.
\end{corollary}
\begin{proof}
We have already shown that $V^{\varepsilon, *}(t,x)=V^{\varepsilon, \#}(t,x)$ is a viscosity subsolution while $V_{*}^{\varepsilon}(t,x)=V^{\varepsilon}(t,x)$ is a viscosity supersolution of \eqref{hver} with initial condition $g$. For Lemma \ref{5l} we know that $V^{\varepsilon, \#}(t,x)\leq g(x)$ and $V(T,x)=g(x)$ so, by comparison principle, it holds $V^{\varepsilon, \#}(t,x)\leq V^{\varepsilon}(t,x)$. By definition of $\limsup$ we have $V^{\varepsilon, \#}(t,x)\geq V^{\varepsilon}(t,x)$ i.e. $V^{\varepsilon}(t,x)$ is upper semicontinuous. Since $V^{\varepsilon}(t,x)$ is also lower semicontinuous we can conclude immediately stating that $V^{\varepsilon}(t,x)$ is continuous.
\end{proof}

\end{document}